\documentclass[11pt,oneside]{article}
\usepackage[english]{babel}
\usepackage{amssymb}
\usepackage[margin = 1.2in]{geometry}
\usepackage{amsthm}
\usepackage{amsmath}
\usepackage{a4wide}
\usepackage{esvect}
\usepackage{setspace}
\usepackage{hyperref}
\usepackage{verbatim}
\usepackage[dvips]{graphicx}
\usepackage{t1enc}
\usepackage{color}
\usepackage{geometry}
\usepackage{todonotes}
\newtheorem{thm}{Theorem}
\newtheorem{claim}[thm]{Claim}

\newtheorem{lemma}[thm]{Lemma}

\newtheorem{question}[thm]{Question}

\DeclareMathOperator{\pp}{pp}

\newcommand{\mP}{\mathcal{P}}

\newcommand{\ignore}[1]{}

\title{On the maximum diameter of path-pairable graphs}
\date{\vspace{-5ex}}
 \author{Ant\'onio Gir\~ao, G\'abor M\'esz\'aros, Kamil Popielarz, and Richard Snyder}

\begin{document}
\maketitle
\singlespace
\begin{abstract}

    \setlength{\parskip}{\medskipamount}
    \setlength{\parindent}{0pt}
    \noindent
    
    A graph is \emph{path-pairable} if for any pairing of its vertices there exist edge disjoint paths joining the vertices in each pair.
    We obtain sharp bounds on the maximum possible diameter of path-pairable graphs which either have a given number of edges, or are $c$-degenerate.
    Along the way we show that a large family of graphs obtained by blowing up a path is path-pairable, which may be of independent interest.

\end{abstract}

\section{Introduction}
\emph{Path-pairability} is a graph theoretical notion that emerged from a practical networking problem introduced by Csaba, Faudree, Gy\'arf\'as, Lehel, and Schelp \cite{CS}, and further studied by Faudree, Gy\'arf\'as, and Lehel \cite{mpp,F,pp} and by Kubicka, Kubicki and Lehel \cite{grid}. Given a fixed integer $k$ and a simple undirected graph $G$ on at least $2k$ vertices, we say that $G$ is {\it $k$-path-pairable} if, for any pair of disjoint sets of distinct vertices $\{x_1,\dots,x_k\}$ and $\{y_1,\dots,y_k\}$ of $G$, there exist $k$ edge-disjoint paths $P_1,P_2,\dots,P_k$, such that $P_i$ is a path from $x_i$ to $y_i$, $1\leq i\leq k$. The path-pairability number of a graph $G$ is the largest positive integer $k$, for which $G$ is $k$-path-pairable, and it is denoted by $\pp(G)$. A $k$-path-pairable graph on $2k$ or $2k+1$ vertices is simply said to be {\it path-pairable}.

Path-pairability is related to the notion of \textit{linkedness}. A graph is $k$-\emph{linked} if for any choice of $2k$ vertices $\{s_1, \ldots , 
s_k, t_1, \ldots , t_k\}$ (not necessarily distinct), there are internally vertex disjoint paths $P_1, \ldots , P_k$ with $P_i$ joining $s_i$ to $t_i$ for $1 \le i \le k$. Bollob{\'a}s and Thomason~\cite{BollobasThomason} showed that any $2k$-connected graph with a lower bound on its edge density is $k$-linked. On the other hand, a graph being path-pairable imposes no constraint on the connectivity or edge-connectivity of the graph. The most illustrative examples of this phenomenon are the stars $K_{1, n-1}$. Indeed, it is easy to see that stars are path-pairable, while they are neither $2$-connected nor $2$-edge-connected. Note that, for any pairing of the vertices of $K_{1, n-1}$, joining two vertices in a pair is straightforward due to the presence of a vertex of high degree, and the fact that the diameter is small. This example motivates the study of two natural questions about path-pairable graphs: given a path-pairable graph $G$ on $n$ vertices, how small can its maximum degree $\Delta(G)$ be, and how large can its diameter $d(G)$ be? This note addresses some aspects of the second question. To be precise, for a family of graphs $\mathcal{G}$ let us define $d(n, \mathcal{G})$ as follows:
\[
	d(n, \mathcal{G}) = \max\{d(G): G \in \mathcal{G} \text{ and } G \text{ is path-pairable on } n \text{ vertices}\}.
\]
When $\mathcal{G}$ is the family of path-pairable graphs, we shall simply write $d(n)$ instead of $d(n, \mathcal{G})$. 

The maximum diameter of arbitrary path-pairable graphs was investigated by M\'esz\'aros \cite{me_diam} who proved that $d(n) \le 6 \sqrt{2} \sqrt{n}$. 
Our aim in this note is to investigate the maximum diameter of path-pairable graphs when we impose restrictions on the number of edges and on how the edges are distributed. 
To state our results, let us denote by $\mathcal{G}_m$ the family of graphs with at most $m$ edges. 
The following result determines $d(n, \mathcal{G}_m)$ for a certain range of $m$. 

\begin{thm}\label{diam_m}
    If $2n \le m \le \frac{1}{4}n^{3/2}$ then
    \[
        \sqrt[3]{\frac{1}{2}m-n} \le d(n, \mathcal{G}_m) \le 16 \sqrt[3]{m}.
    \]
\end{thm}
We remark that the upper bound in the Theorem~\ref{diam_m} holds for $m$ in any range, but when $m \ge \frac{1}{4}n^{3/2}$ the bound obtained by M\'esz\'aros \cite{me_diam} is sharper. 
Determining the behaviour of the maximum diameter among path-pairable graphs on $n$ vertices with fewer than $2n$ edges remains an open problem.
In particular, we do not know if the maximum diameter must be bounded  (see Section \ref{sec:final}).

Following this line of research, it is very natural to consider the problem of determining the maximum attainable diameter for other classes of graphs. For example, what is the behaviour of the maximum diameter of path-pairable \emph{planar} graphs? Although we could not give a satisfactory answer to this particular question, we were able to do so for graphs which are $c$-\emph{degenerate}. 
As usual, we say that an $n$-vertex graph $G$ is $c$-\emph{degenerate} if there exists an ordering $v_1,\ldots,v_n$ of its vertices such that $|\{v_j: j > i, v_iv_j\in E(G) \}|\leq c$ holds for all $i=1,2,\ldots,n$. We let $\mathcal{G}_{c\text{-deg}}$ denote the family of $c$-degenerate graphs. Clearly all $c$-degenerate graphs have a linear number of edges, so Theorem~\ref{diam_m} implies that $d(n, \mathcal{G}_{c\text{-deg}}) = O(\sqrt[3]{n})$. However, as the next result shows, this bound is far from the truth.

\begin{thm}\label{diam_cdeg}
Let $c \ge 5$ be an integer. Then
\[
	(2+o(1)) \frac{\log(n)}{\log({\frac{c}{c-2}})}  \leq d(n, \mathcal{G}_{c\text{-deg}}) \leq  (12+o(1)) \frac{\log(n)}{\log(\frac{c}{c-2})}
\]	
	as $n \rightarrow \infty$. 
\end{thm}

We remark that we have not made an effort to optimize the constants appearing in the upper and lower bounds of Theorems~\ref{diam_m} and~\ref{diam_cdeg}.
\subsection{The Cut-Condition}

While path-pairable graphs need not be highly connected or edge-connected, they must satisfy certain `connectivity-like' conditions that we shall need in the remainder of the paper. We say a graph $G$ on $n$ vertices satisfies the \emph{cut-condition} if for every $X \subset V(G)$, $|X| \le n/2$, there are at least $|X|$ edges between $X$ and $V(G)\setminus X$. Clearly, a path-pairable graph has to satisfy the cut-condition. On the other hand, satisfying the cut-condition is not sufficient to guarantee path-pairability in a graph; see \cite{me_pp} for additional details.

\subsection{Organization and Notation}
The proofs of the lower bounds in Theorems~\ref{diam_m} and~\ref{diam_cdeg} require constructions of path-pairable graphs with large diameter. In Section~\ref{sec:blowup}, we 
show how to obtain such graphs by proving that a more general class of graphs is path-pairable. In Sections~\ref{sec:proofdiam_m} and~\ref{sec:proofdiam_cdeg} we shall complete the proofs of
Theorems~\ref{diam_m} and~\ref{diam_cdeg}, respectively. Finally, we mention some open problems in Section~\ref{sec:final}.

Our notation is standard. Thus, for a (simple, undirected) graph $G$ we shall denote by $V(G)$ and $E(G)$ the vertex set and edge set of $G$, respectively. We also let $|G|$ and $d(G)$ denote the number of vertices and diameter of $G$, respectively. For a vertex $x \in V(G)$ we let $N_{G}(x)$ denote the neighbourhood of $x$ in $G$, and we shall omit the subscript `$G$' when no ambiguity arises. 
\section{Path-pairable graphs from blowing up paths}\label{sec:blowup}
   
In this section, we will show how to construct a quite general class of graphs which have high diameter and are path-pairable. 
Let $G$ be a graph with vertex set $V(G)=\{v_1,\ldots,v_k\}$, and let $G_1,\ldots, G_k$ be graphs. We define the \textit{blown-up graph} $G(G_1,\ldots, G_k)$ as follows:  replace every vertex $v_i$ in $G$ by the corresponding graph $G_i$, and for every edge $v_iv_j \in E(G)$ insert a complete bipartite graph between the vertex sets of $G_i$ and $G_j$.

Let $P_k$ denote the path on $k$ vertices. The following lemma asserts that if we blow-up a path with graphs $G_1, \ldots , G_k$, such that $G_{i}$ is path-pairable for $i \le k-1$, and certain properties inherited from the cut-condition hold, then the resulting blow-up is path-pairable. 

\begin{lemma}\label{blown-up}
    Suppose that $G_1, \ldots ,G_k$ are graphs on $n_1, \ldots , n_k$ vertices, respectively, where $G_{i}$ is path-pairable for $i \le k-1$.
    Let $n = \sum_{i =1}^k n_i$ and let $u_i = \sum_{j=1}^i n_j$ for $i=1, \dots , k-1$. Then $P_k(G_1,\ldots, G_k)$ is path-pairable if and only if 
\begin{equation}\label{eq_1}
    n_i\cdot n_{i+1}\geq \min(u_i,n-u_i)
\end{equation}
holds for $i=1,\ldots,k-1$.
\end{lemma}
\begin{proof}
For each $i = 1, \ldots , k$, let $U_i = \bigcup_{j=1}^i V(G_j)$ so that $u_i = |U_i|$. Now, if $P_k(G_1, \ldots , G_k)$ is path-pairable, then we may apply the cut-condition to the cut $\{U_i, V(G)\setminus U_i\}$. This implies $n_i\cdot n_{i+1}\geq \min(u_i,n-u_i)$ must hold for $i=1,\ldots , k-1$. In the remainder, we show that this simple condition is enough to yield the path-pairability of 
$G := P_k(G_1, \ldots , G_k)$. Assume that a pairing $\mathcal{P}$ of the vertices of $G$ is given. If $\{u, v\} \in \mathcal{P}$ we shall say that $u$ is a \emph{sibling} of $v$ (and vice-versa).
We shall define an algorithm that sweeps through the classes $G_1,G_2,\ldots, G_k$ and joins each pair of siblings via edge-disjoint paths. 

First we give an overview of the algorithm.
We proceed by first joining pairs $\{u, v\} \in \mathcal{P}$ via edge-disjoint paths such that $u$ and $v$ belong to different $G_i$'s, and then afterwards joining pairs that remain inside some $G_j$ (using the path-pairability of $G_j$). 
Before round $1$ we use the path-pairability property of $G_{1}$ to join those siblings which belong to $G_{1}$.
In round $1$ we assign to every vertex $u$ of $G_1$ a vertex $v$ of $G_2$. 
If $\{u, v\} \in \mathcal{P}$ are siblings, then we simply choose the edge $uv$. 
Then we join the siblings which are in $G_{2}$ again using the path-pairability property of $G_{2}$.
For those paths $uv$ that have not ended (because $\{u, v\} \notin \mathcal{P}$) we shall continue by choosing a new vertex $w$ in $G_3$ and continue the path with edge $vw$, and so on. 
Paths which have not finished joining a pair of siblings we shall call \emph{unfinished}; otherwise, we say the path is \emph{finished}. 
The last edge which completes a finished path we shall call a \emph{path-ending edge}. 
During round $i$ we shall first choose those vertices in $G_{i+1}$ which, together with some vertex of $G_i$, form path-ending edges.
At the end of round $i$, in $G_{i+1}$ we will have endpoints of unfinished paths and perhaps also some endpoints of finished paths. 
Note that the vertices of $G_{i+1}$ might be endpoints of several unfinished paths. 
For $x \in G_{i+1}$ let $w(x)$ denote the number of unfinished paths $P\cup \{x\}$ with $P \subset U_i$ at the end of round $i$ which are to be extended by a vertex of $G_{i+2}$ (including the single-vertex path $x$ in the case when $x$ was not joined to its sibling in the latest round). 
Note that every such path corresponds to a yet not joined vertex in $U_{i+1}$ as well as to another vertex yet to be joined lying in $V(G)\setminus U_{i+1}$. 
It follows that
\begin{equation}\label{eq:weights}
	\sum_{x \in G_{i+1}}w(x) \le \min(u_{i+1}, n-u_{i+1}).
\end{equation}

Let us now be more explicit in how we make choices in each round.
We shall maintain the following two simple conditions throughout our procedure (the first of which has been mentioned above):
\begin{itemize}
\item[(a)] During round $i$ ($1\le i \le k-1$), if $w \in G_i$ is the current endpoint of the path which began at some vertex $u\in U_i$ (possibly $u=w$), and $\{u, v\} \in \mathcal{P}$ for $v \in G_{i+1}$, then we join $w$ to $v$. Informally, we choose path-ending edges when we can.
\item[(b)] $w(x) \leq n_{i+1}$ for all $x \in G_i$, for $i =1, \ldots , k-1$. 
\end{itemize}
The second condition above is clearly necessary in order to proceed during round $i$, as ${|N(x)\cap G_{i+1}| = n_{i+1}}$ for every $x \in G_i$, and hence we cannot continue more than $n_{i+1}$ unfinished paths through $x$. 

We claim that as long as both of the above conditions are maintained, the proposed algorithm finds a collection of edge-disjoint paths joining every pair in $\mathcal{P}$.
Both conditions are clearly satisfied for $i=1$ as $w(x) \le 1\leq n_2$ for all $x\in G_1$. 
Let $i \ge 2$ and suppose both conditions hold for rounds $1, \ldots , i-1$. 
Our aim is show that an appropriate selection of edges between $G_i$ and $G_{i+1}$ exists in round $i$ to maintain the conditions. 
We start round $i$ by choosing all path-ending edges with endpoints in $G_i$ and $G_{i+1}$; this can be done since, by induction, $w(x) \le n_{i+1}$ for every $x \in G_i$. 
Observe that if $i = k-1$ then the only remaining siblings are in $G_{k}$.
Then for every $\left\{ u, v \right\} \in \mathcal{P}$ such that $u,v \in G_{k}$ we can find a vertex $w$ in $G_{k-1}$ and join $u, v$ with the path $uwv$.
When $i < k-1$ then the remaining paths can be continued by assigning arbitrary vertices from $G_{i+1}$ (without using any edge multiple times). 
We choose an assignment that balances the `weights' in $G_{i+1}$. 
More precisely, let us choose an assignment of the vertices that minimizes
\[
	\sum\limits_{a\in G_{i+1}}w(a)^2.
\]
If for every $x \in G_{i+1}$ we have that $w(x) \le n_{i+2}$ we are basically done.
It remains to find edge-disjoint paths inside $G_{i+1}$ for those pairs $\{x, y\} \in \mathcal{P}$ whose vertices belong to $G_{i+1}$.
But this is possible because of the assumption that $G_{i+1}$ is path-pairable.

Suppose then that in the above assignment there exists $x \in G_{i+1}$ with $w(x) \geq n_{i+2}+1$. 
We first claim that, under this assignment, no other vertex of $G_{i+1}$ has small weight.
\begin{claim}\label{claim:y-big}
Every vertex $y\in G_{i+1}$ satisfies $w(y)\geq n_{i+2}-1$. 
\end{claim}
\begin{proof}
Suppose there is $y \in G_{i+1}$ such that $w(y) \le n_{i+2}-2$. 
Then, as $w(x)>w(y)+2$, there exist vertices $v_1,v_2\in G_i$ such that certain paths ending at $v_1$ and $v_2$ were joined in round $i$ to $x$ ($x$ was assigned as the next vertex of these paths) but no paths at $v_1$ or $v_2$ were assigned $y$ as their next vertex. 
Observe that at least one of the edges $v_1x$ and $v_2x$ is not a path ending edge which could have been replaced by the appropriate $v_1y$ or $v_2y$ edge, respectively. 
That operation would result in a new assignment with a smaller square sum $\sum_{a\in G_{i+1}}w(a)^2$, which is a contradiction.
\end{proof}

Therefore, we may assume $w(y)\geq n_{i+2} - 1$ for all $y\in G_{i+1}$. 
In this case, partition the vertices of $G_{i+1}$ into three classes: 
\begin{align*}
X &= \{v\in G_{i+1}: w(v) \geq n_{i+2} +1\}\\
Y &= \{v\in G_{i+1}: w(v) = n_{i+2} - 1\} \\
Z &= \{v\in G_{i+1}: w(v) = n_{i+2}\}.
\end{align*}

Observe first that $1 \le |X| \leq |Y|$, since otherwise using~(\ref{eq:weights}) we have \[n_{i+1}n_{i+2}+1\leq \sum\limits_{s\in G_{i+1}}w(s)\leq \min(u_{i+1},n-u_{i+1}),\] contradicting condition~(\ref{eq_1}).
Notice also that the same argument as in Claim~\ref{claim:y-big} shows that $w(v) \le n_{i+2}+1$ for every $v \in G_{i+1}$, hence we can actually write 
\[
    X = \left\{ v \in G_{i+1}: w(v) = n_{i+2}+1 \right\}.
\]

We will need the following claim which asserts that if there are siblings in $G_{i+1}$ then they must belong to $Z$.
\begin{claim}\label{claim:Z_pairs}
     If $\{u,v\} \in \mathcal{P}$ and $u, v \in G_{i+1}$, then $u, v \in Z$.
\end{claim}
\begin{proof}
We first show that every $y \in Y$ is incident to a path-ending edge. Suppose, to the contrary, that there is $y \in Y$ such that there is no path-ending edge which ends at $y$.
    It follows that there are at most $w(y)$ vertices in $G_{i}$ which had been joined to $y$.
    Hence we can take any $x \in X$ and find $z \in G_{i}$ which was not joined to $y$, and such that $xz$ is not a path-ending edge.
    Replacing $zx$ by $zy$ would result in a smaller square sum $\sum_{a \in G_{i+1}}w(a)^{2}$, which gives a contradiction.
    
    Now, let $\{u, v\} \in \mP$ such that $u, v \in G_{i+1}$. 
    Since every $y \in Y$ is incident to a path-ending edge, we have that $u, v \not\in Y$.
    Suppose, for contradiction, that $u \in X$.
    Then $u$ was joined to $w(u) = n_{i+2}+1$ vertices in $G_{i}$, and hence for every $y \in Y$, there is $z \in G_{i}$ which was joined to $u$ but not $y$.
    Replacing $zu$ by $zy$ would result in a smaller square sum $\sum_{a \in G_{i+1}}w(a)^{2}$, which again gives a contradiction.
\end{proof}

Finally, we shall show that we can reduce the weights of the vertices in $X$ (and pair the siblings inside $G_{i+1}$) using the path-pairable property of $G_{i+1}$.
For every $x \in X$ pick a different vertex $y_{x} \in Y$ (which we can do, since $|Y| \ge |X|$) and let $\mathcal{P'} = \left\{ \{u, v\} \in \mathcal{P} : u, v \in G_{i+1}  \right\} \cup \left\{ \{x, y_{x}\} : x \in X  \right\}$.
Since $G_{i+1}$ is path-pairable, we can find edge-disjoint paths joining the siblings in $\mathcal{P'}$ (note that by Claim~\ref{claim:Z_pairs} none of the pairs $\{x, y_x\}$ interfere with any siblings $\{u, v\} \in \mP$ with $u, v \in G_{i+1}$).
Observe now that for every $x \in X$ one path has been channeled to a vertex $y\in Y$, thus the number of unfinished path endpoints at $x$ has dropped to $n_{i+2}$ and so the condition is maintained. 

\end{proof}

    We close the section by pointing out that the condition that the graphs $G_{i}$ are path-pairable is necessary.
    We do this by giving an example of a blown-up path $P_k(G_{n_1},\ldots,G_{n_k})$ that satisfies the cut-conditions of Lemma~\ref{blown-up} yet it is not path-pairable unless some of $G_i$'s are path-pairable as well. 
    For the sake of simplicity we set $k = 5$ and prove that $G_3$ has to be path-pairable. 
    Let $n = 2t^2 + t$ for some even $t \in \mathbb{N}$ and let $n_1 = n_5 = t^2-t$, $n_2=n_3 = n_4 = t$. 
    Clearly $P_5(G_{n_1},\ldots,G_{n_5})$ satisfies the Condition~\ref{eq_1} of Lemma~\ref{blown-up}. 
    Observe, that any pairing of the vertices in $G_{1} \cup G_{2}$ with the vertices in $G_{4} \cup G_{5}$ has to use all the edges between $G_{3}$ and $G_{2} \cup G_{4}$.
    Therefore if we additionally pair the vertices inside $G_{3}$, then the paths joining those vertices can only use the edges in $G_{3}$, therefore $G_{3}$ has to be path-pairable.

\section{Proof of Theorem \ref{diam_m}}\label{sec:proofdiam_m}

    Take $x,y\in V(G)$ such that $d(x,y) = d(G)$ and let $V_{i}$ be the set of vertices at distance exactly $i$ from $x$, for every $i$.
    Observe that $V_0 = \{x\}$ and $y\in V_{d(G)}$. 
    For $i \in \left\{ 1, \dots, d(G) \right\}$ define $n_{i}$ to be the size of $V_{i}$ and let $u_{i} = \sum_{j=0}^{i} n_{j}$.

    We need the following claim.
    \begin{claim}
        $u_{2k+1} \geq \binom{k+2}{2}$ as long as $u_{2k+1}\leq\frac{n}{2}$.
    \end{claim}
    \begin{proof}
        We shall use induction on $k$.
        For $k = 0$ it is clear.
        Assume that $u_{2k-1} \ge \binom{k+1}{2}$.
        By the cut-condition we have that the number of edges between $V_{2k}$ and $V_{2k+1}$ is at least $u_{2k-1}$, hence $n_{2k}\cdot n_{2k+1} \ge u_{2k-1} \ge \binom{k+1}{2}$. 
        By the arithmetic-geometric mean inequality, $n_{2k} + n_{2k+1} \ge 2\sqrt{\binom{k+1}{2}} \ge k+1$.
        As $u_{2k+1} = u_{2k-1} + n_{2k} + n_{2k+1}$, we have $u_{2k+1} \ge \binom{k+2}{2}$.
    \end{proof}

    Now, let $A = \bigcup_{i=0}^{\lfloor d / 3 \rfloor}V_{i}$, $B = \bigcup_{i = \lfloor d / 3 \rfloor +1}^{2d/3} V_{i}$, $C = \bigcup_{i = \lfloor 2d/3 \rfloor + 1}^{d} V_{i}$.
    Observe, that $|A|, |C| \ge \min\left\{\frac{n}{2}, \frac{d^{2}}{100}\right\}$, so joining vertices in $A$ with vertices in $C$ requires at least $ \min\left\{\frac{n}{2},\frac{d^2}{100}\right\}\cdot\frac{d}{3}$ edges.
    Hence,
    \[ \min\left\{\frac{n}{2},\frac{d^2}{100}\right\}\cdot\frac{d}{3}\leq m,\]
    which implies
    \[d\leq \max\left\{\frac{6m}{n},16\sqrt[3]{m}\right\}.\]
    Notice that whenever $m \le 4n^{3/2}$ we have $d \le 16\sqrt[3]{m}$.
    Let us remark that if $m \ge \frac{1}{4}n^{3/2}$ then the upper bound is trivially satisfied by the general upper bound obtained in \cite{me_diam}. 

    For the lower bound, let $n$ and $2n \le m \le \frac{1}{4}n^{3/2}$ be given. 
    For any natural number $\ell$ we shall denote by $S_\ell$ the star $K_{1, \ell-1}$ on $\ell$ vertices.
    Consider the graph $G = P_{k}(G_{1},\dots,G_{k})$ on $n$ vertices,
    where $k = \left\lfloor \sqrt[3]{\frac{m}{2}-n} \right\rfloor$ and $G_{1} = G_{2} = \dots = G_{k} = S_{k}$, $G_{k+1} = S_{k^{2}}$, $G_{k+2} = S_{2}$, and $G_{k+3}$ is an empty graph on $n-2k^{2}-2$ vertices.

    Straightforward calculation shows that $u_i = i\cdot k$, for $i\leq k, u_{k+1}= 2k^2$, and $u_{k+2} = 2k^2 + 2$.
    Also $n_1n_2=n_2n_3=\ldots=n_{k-1}n_{k}=k^2$, $n_{k}n_{k+1}=k^3$, $n_{k+1}n_{k+2}=2k^2$, and $n_{k+2}n_{k+3}=2n-4k^2-4$.
    Therefore, for $i \in \left\{ 1, \dots, k+1 \right\}$ we have $n_{i} \cdot n_{i+1} \ge u_{i} \ge \min(u_{i}, n-u_{i})$ and $n_{k+2} \cdot n_{k+3} \ge n_{k+3} \ge \min(u_{k+2}, n-u_{k+2})$.
    Hence it follows from Lemma~\ref{blown-up} that $G$ is path-pairable.

    It is easy to check that the number of edges in $G$ is at most $2n + 2k^{3} \le m$.
    On the other hand, the diameter of $G$ is $k+2 \ge \sqrt[3]{\frac{m}{2} - n}$.

\section{Proof of Theorem \ref{diam_cdeg}}\label{sec:proofdiam_cdeg}

    In this section, we investigate the maximum diameter a path-pairable $c$-degenerate graph on $n$ vertices can have. 
    We shall assume that $c$ is an integer and $c\geq 5$. 

    Let $G$ be a $c$-degenerate graph on $n$ vertices with diameter $d$.
    We shall show first that $d \le 4\log_{\frac{c+1}{c}}(n)+3$.
    Let $x \in G$ be such that there is $y \in G$ with $d(x,y) = d$. 
    For $i \in \left\{ 0, \dots, d \right\}$, write $V_{i}$ for the set of vertices at distance $i$ from $x$.
    Let $n_{i} = |V_{i}|$ and $u_{i} = \sum_{j = 0}^{i}n_{j}$.
    Observe that $|V_{i}| \ge 1$ for every $i \in \left\{ 0, \dots, d \right\}$.
    We can assume that $u_{\lfloor \frac{d}{2} \rfloor} \le \frac{n}{2}$ (otherwise we repeat the argument below with $V'_{i} = V_{d-i})$.

    The result will easily follow from the following claim.
    \begin{claim}
        $u_{2k+1} \ge \left( \frac{c+1}{c} \right)^{k}$ as long as $u_{2k+1} \le \frac{n}{2}$.
    \end{claim}
    Let us assume the claim and prove the result. 
    Letting $k = \frac{\lfloor\frac{d}{2}\rfloor -1}{2}$, we have that $n/2 \ge u_{2k+1} \ge \left( \frac{c+1}{c} \right)^{\frac{\lfloor \frac{d}{2} \rfloor-1}{2}}$.
    Hence $d \leq 4\log_{\frac{c+1}{c}}(n)+3 = 4 \frac{\log(n)}{\log(\frac{c+1}{c})} + 3 \le 4 \frac{\log(n)}{\log(\frac{c}{c-2})}\frac{\log(\frac{c}{c-2})}{\log(\frac{c+1}{c})} + 3\le 12 \frac{\log(n)}{\log(\frac{c}{c-2})}+3$, where the last inequality follows from the easy to check fact that $\frac{\log(\frac{c}{c-2})}{\log(\frac{c+1}{c})} \le 3$, for all $c \ge 5$.
    \begin{proof}[Proof of the Claim]
        We shall prove the claim by induction on $k$.
        The base case when $k=0$ is trivial as $u_{1} \ge 2$.
        Suppose the claim holds for every $l \le k-1$.
        Since $G$ is $c$-degenerate we have that $e(V_{2k}, V_{2k+1}) \le c\left( n_{2k} + n_{2k+1} \right)$.
        On the other hand, it follows from the cut-condition that $e(V_{2k}, V_{2k+1}) \ge u_{2k} = u_{2k-1}+ n_{2k}$.
        Therefore, by the induction hypothesis, we have 
        $n_{2k} + n_{2k+1} \ge
        \frac{1}{c}\left( u_{2k-1}+n_{2k} \right) 
        \ge \frac{1}{c} \left( \left( {\frac{c+1}{c}}\right)^{k-1} +n_{2k} \right)
        \ge \frac{1}{c}\left( \frac{c+1}{c} \right)^{k-1}$. Hence, $u_{2k+1}=u_{2k-1}+n_{2k}+n_{2k+1}\geq \left( \frac{c+1}{c} \right)^{k-1}+ \frac{1}{c}\left( \frac{c+1}{c} \right)^{k-1}= \left( 1+\frac{1}{c} \right) \left( \frac{c+1}{c} \right)^{k-1} \ge \left( \frac{c+1}{c} \right)^{k}$, which proves the claim.
    \end{proof}
    We shall prove the lower bound assuming $c$ is an odd integer; when $c$ is even we apply the same argument for $c-1$.  
    
    To do so, consider the graph $G = P(G_{1},\ldots, G_{2m^{\prime}-1})$ for some $m^{\prime}\in \mathbb{N}$, which we specify later. 
    Firstly, we shall define the sizes of $G_i$ for $i \in \{1,\ldots, 2m'-1\}$. 
    To do so, let us define a sequence  $\{n_i\}_{i\in\mathbb{N}}$ where $n_{2i} =\frac{c-1}{2} $ and $n_{2i+1}$ is defined recursively in the following way: 

    \begin{equation}
        n_{2i+1}= \left \lceil \frac{2}{c-1} \cdot \sum_{j=1}^{2i} n_j \right \rceil \leq \frac{2}{c-1}\sum_{j=1}^{2i} n_j +1
    \end{equation}
    
    Let $m$ be the largest integer such that  $\sum_{j=1}^{m} n_j \leq n/2$. 
    We let $m'=m$ when $m$ is odd  and $m'=m-1$ when $m$ is even. 
    Moreover, let $|G_{m'}|=n-2\sum_{j=1}^{j=m'-1} n_j$ and let $|G_i|=n_i$ for $1\leq i < m'$ and $|G_{m'+j}|= |G_{m'-j}|$ for $ j \in \{1,\ldots, m'-1\}$. 

    For all $i\in \{1,\ldots 2m'-1\}$ let $G_i = S_{n_i}$ be a star on $n_i$ vertices. 
    It is easy to check that the graph $P_{2m'-1}(G_1,\ldots,G_{2m'-1}) $ is path-pairable by Lemma \ref{blown-up}.  
    It has diameter at least $2m-4$ and $m \geq \log_{\frac{c+1}{c-1}}(n)(1+o(1))$. 
    Again an easy verification shows that the graph $G$ is $c$-degenerate.

\section{Final remarks and open problems}\label{sec:final}

    We obtained tight bounds on the parameter $d(n, \mathcal{G}_{m})$ when $(2+\epsilon)n \leq m \leq  \frac{1}{4} n^{3/2}$, for any fixed $\epsilon >0$. 
    It is an interesting open problem to investigate what happens when the number of edges in a path-pairable graph on $n$ vertices is around $2n$. We ask the following:
    \begin{question}
        Is there a function $f$ such that for every $\epsilon>0$ and for every path-pairable graph $G$ on $n$ vertices with at most $(2-\epsilon)n$ edges, the diameter of $G$ is bounded by $f(\epsilon)$?
    \end{question}

    Another line of research concerns determining the behaviour of $d(n, \mP)$, where $\mP$ is the family of planar graphs. Since planar graphs are $5$-degenerate, it follows from Theorem~\ref{diam_cdeg} that the diameter of a path-pairable planar graph on $n$ vertices cannot be larger than $c \log{n}$.
    This fact makes us wonder whether there are path-pairable planar graphs with unbounded diameter.
    \begin{question}
        Is there a family of path-pairable planar graphs with arbitrarily large diameter?
    \end{question}    
    
    The graph constructed in the proof of the lower bound in Theorem~\ref{diam_cdeg} when $c=5$ is not planar since it contains a copy of $K_{3,3}$. Therefore, it cannot be used to show that the diameter of a path-pairable planar graph can be arbitrarily large (note, however, that this graph does not contain a $K_7$-minor nor a $K_{6,6}$-minor). We end by remarking that we were able to construct an infinite family of path-pairable planar graphs with diameter $6$, but not larger.

\bibliographystyle{acm}
\bibliography{PPPgraphs_diam.bib}

\begin{thebibliography}{1}

\bibitem{BollobasThomason}
{\sc Bollob{\'a}s, B., and Thomason, A.}
\newblock Highly linked graphs.
\newblock {\em Combinatorica 16}, 3 (1996), 313--320.

\bibitem{CS}
{\sc Csaba, L., Faudree, R., Gy\'arf\'as, A., Lehel, J., and Schelp, R.}
\newblock Networks communicating for each pairing of terminals.
\newblock {\em Networks 22\/} (1992), 615--626.

\bibitem{mpp}
{\sc Faudree, R., Gy\'arf\'as, A., and Lehel, J.}
\newblock Minimal path pairable graphs.
\newblock {\em Congressus Numerantium 88\/} (1992), 111--128.

\bibitem{F}
{\sc Faudree, R.~J.}
\newblock Properties in path-pairable graphs.
\newblock {\em New Zealand Journal of Mathematics 21\/} (1992), 91--106.

\bibitem{pp}
{\sc Faudree, R.~J., Gy\'arf\'as, A., and Lehel, J.}
\newblock Path-pairable graphs.
\newblock {\em Journal of Combinatorial Mathematics and Combinatorial Computing
  20\/} (1999), 145--157.

\bibitem{grid}
{\sc Kubicka, E., Kubicki, G., and Lehel, J.}
\newblock Path-pairable property for complete grids.
\newblock {\em Combinatorics, Graph Theory, and Algorithms 1\/} (1999),
  577--586.

\bibitem{me_diam}
{\sc M\'esz\'aros, G.}
\newblock Note on the diameter of path-pairable graphs.
\newblock {\em Discrete Mathematics 337 (2014)\/}, 83--86.

\bibitem{me_pp}
{\sc M\'esz\'aros, G.}
\newblock On path-pairability in the cartesian product of graphs.
\newblock {\em Discussiones Mathematicae Graph Theory 36 (2016)\/}, 743--758.

\end{thebibliography}
\end{document}